\documentclass[a4paper,10pt]{article}

\usepackage{times}
\usepackage[T1]{fontenc}
\usepackage[latin1]{inputenc}
\usepackage{amsmath,amssymb}
\usepackage{amsthm}
\usepackage{amscd}

\newtheorem{theorem}{Theorem}[section]
\newtheorem{lemma}[theorem]{Lemma}
\newtheorem{corollary}[theorem]{Corollary}
\newtheorem{proposition}[theorem]{Proposition}

\newtheorem{definition}[theorem]{Definition}

\theoremstyle{definition}
\newtheorem{example}[theorem]{Example}
\newtheorem{remark}[theorem]{Remark}

\numberwithin{table}{section}
\numberwithin{equation}{section}

\begin{document}
\title{Polynomial as a new variable -  a Banach algebra  with a functional calculus} 
\author{ Olavi Nevanlinna }
\maketitle

 \begin{center}
{\footnotesize\em 
Aalto University\\
Department of Mathematics and Systems Analysis\\
 email: Olavi.Nevanlinna\symbol{'100}aalto.fi\\[3pt]
}
\end{center}
 
\begin{abstract}
Given any square matrix or a bounded  operator $A$ in a Hilbert space   such that $p(A)$ is normal (or similar to normal),  we construct a Banach algebra, depending on the polynomial $p$,  for  which a simple functional calculus holds.  When the polynomial  is of degree $d$, then  the algebra deals with continuous  $\mathbb C^d$-valued functions, defined on the spectrum of  $p(A)$. In particular,  the calculus provides a natural approach to deal with nontrivial Jordan blocks and  one does not need  differentiability at such eigenvalues.

\end{abstract}
\bigskip
 
 \section{Introduction}
 
There are many situations in which it  would be desirable to be able to treat  polynomials as  new {\it global} variables.  For example, by Hilbert's lemniscate theorem (see e.g [12])  polynomials can be used to map complicated sets of the complex plane onto discs.  As polynomials are not one-to-one we represent scalar functions in the original variable by a vector valued function in the polynomial.  This leads to {\it multicentric holomorphic calculus} [9].  In [10] we applied it to generalize the von Neumann theorem on contractions in Hilbert spaces.   In such applications one would, given a bounded operator $A$, search for a polynomial $p$ such that  $p(A)$ has a small norm - thus mapping a potentially complicated spectrum into a small disc.

In this paper we study multicentric calculus without assuming the functions to be analytic.  As an application  we consider  situations in which $p(A)$ is diagonalizable  or similar to normal.  Thus, the aim is to remove the Jordan blocks by moving from $A$ to $p(A)$.  To illustrate the goal consider  finite dimensional matrices. If $D ={\rm diag}\{\alpha_j\}$ is a diagonal matrix  and $\varphi$ is a continuous function, then  any reasonable functional calculus satisfies  $\varphi(D)= {\rm diag} \{\varphi(\alpha_j)\}$. Further,  if $A$ is diagonalizable so that with a similarity $T$ we have  $A=T DT^{-1}$,  then we of course set
 \begin{equation}
 \varphi(A)= T \varphi(D)T^{-1}.
 \end{equation}
 However, if $A$ has an eigenvalue with a nontrivial Jordan block, then  the customary approach is to assume that  $\varphi$ is smooth enough at the  eigenvalues so that the off-diagonal elements can be represented by derivatives of $\varphi$. For example, if
\begin{equation}\label{Jordan block}
 J=\begin{pmatrix}
\alpha & 1 &   \\
 & \alpha   & 1 \\
 & &  \alpha 
\end{pmatrix}
\end{equation}
then
\begin{equation}\label{varJordan}
 \varphi(J)=\begin{pmatrix}
\varphi(\alpha) & \varphi'(\alpha) & \frac{1}{2} \varphi''(\alpha)  \\
 & \varphi(\alpha) & \varphi'(\alpha)\\
 & &  \varphi(\alpha)
\end{pmatrix}.
\end{equation}
A collection of  different ways to define $\varphi(A)$ for matrices  can be found  from  [11],   where Higham, following  Gantmacher [6],   says that  a function $\varphi$ is {\it defined  at the spectrum $\sigma(A) =  \{\alpha_j\}$ if the  values  $\varphi^{(k)}(\alpha_j)$ are known for $0\le k \le n_j$}, where $n_j+1$ are the  powers in the minimal polynomial.  

This has two obvious drawbacks.  First, since it is based on  the Jordan form the  functional calculus  is discontinuous:   for diagonalizable matrices it is given for all continuous functions while it requires existence of derivatives  at eigenvalues with nontrivial Jordan blocks. 
Second, the approach cannot conveniently be extended to infinite dimensional spaces. Recall that there is a natural functional calculus for normal operators which easily extends to operators which are similar to normal. If, however,  an eigenvalue with    a  nontrivial Jordan block would exist in the middle of  a cluster of other eigenvalues, then  one would need to have a way to treat function classes which  are continuous  and additionally have derivatives at that particular eigenvalue.

We shall show in this paper that there is a simple way to parametrize continuous functions  which slow down at those places where some extra smoothness is needed.  And it turns out  that this allows a functional calculus which agrees with the holomorphic functional calculus if applied to holomorphic functions  but  is defined for functions which do not need to be differentiable at any point.

The starting point for the calculus is taking $w=p(z)$ as a new variable.  Since such a  change of variable  is only locally injective we compensate this  by replacing  the scalar function 
$$
\varphi :  z \mapsto \varphi(z) \in \mathbb C
$$
by a vector valued function
$$
f : w \mapsto f(w) \in \mathbb C^d
$$
where $d$ is the degree of the   polynomial $p$.      The {\it multicentric representation of} $\varphi$  is then of the form
\begin{equation}\label{mrep}
\varphi(z) = \sum_{j=1}^d \delta_j(z) f_j(p(z)), 
\end{equation}
where $\delta_j$'s are the Lagrange interpolation polynomials such that
$\delta_j(\lambda_j)=1$ while $\delta_j(\lambda_k)=0$  when $k\not=j,$ [9].

If now $p(A)$ is diagonalizable, one can apply  the  known functional calculus to represent $f_j(p(A))$. But since $\delta_j$'s are polynomials, $\delta_j(A)$ is well defined and differentiability of $\varphi$ is not needed.

The paper is organized as follows.   We first consider the Banach space of continuous functions $f$:  $M\rightarrow \mathbb C^d$ and associate with it a product ,  "polyproduct" $\circledcirc$,  such that it becomes a Banach algebra, which we denote by $C_\Lambda(M)$. Here  $\Lambda$ denotes the set of zeros of the polynomial $p$.  Then the functions $\varphi$ in (\ref{mrep})  can be viewed as Gelfand transformations $\hat f$ of functions  $f\in C_\Lambda(M)$.
Towards the end of the paper we discuss the functional calculus for operators in  Hilbert spaces $H$ such that $p(A)$ is similar to a normal operator. In particular we study  the mapping 
$\chi_A$
which associates to $f$ a  bounded operator  $\chi_A (f) \in \mathcal B(H)$
$$
\chi_A(f) = \sum_{j=1}^d \delta_j(A) f_j(p(A))
$$
and show that we get a homomorphism 
$
\chi_A(f \circledcirc g) = \chi_A(f) \chi_A(g)
$
 which, in an appropriate quotient algebra, satisfies a spectral mapping theorem.


\section{Construction of the Banach algebra}

\subsection{Multicentric representation of functions}

We assume given a  polynomial $p(z)=(z-\lambda_1)\cdots (z-\lambda_d) $ with distinct zeros $\Lambda= \{\lambda_j\}_{j=1}^d$ mapping the $z$-plane  onto $w$-plane:  $w=p(z)$.  In addition we denote by $\Lambda_1 =\{ z \ :  p'(z)=0\}$ the set of critical points of $p$.  We call the points of $\Lambda$ as the local centers of the  multicentric calculus. Recall that by the Gauss-Lucas theorem $\Lambda_1$ is in the convex hull of $\Lambda$.

Suppose  $\delta_j(z)$ are the  Lagrange interpolation polynomials with interpolation points in $\Lambda$ so that 
$$
\delta_j(z) = \frac{p(z)}{p'(\lambda_j)(z-\lambda_j)} = \prod_{k \not=j}\frac{z-\lambda_k}{\lambda_j - \lambda_k}.
$$
Assume then that we are given  a  function $f$ mapping a compact $M\subset \mathbb C$ into $\mathbb C^d$.  It determines a unique function  $\varphi$ on $K=p^{-1}(M)$ if we set
$$
\varphi(z) = \sum_{j=1}^d \delta_j(z) f_j(p(z)) {\text { for } } z\in K.
$$
We say that $\varphi$ is given on $K$ by  a {\it  multicentric representation}  and denote it in short
$$
\varphi = \mathcal L f.
$$
In the reverse direction, suppose we are given a scalar function $\varphi$ on a set $K_0$.  Then a necessary condition for $f$ to be determined uniquely is that $K_0$ is {\it balanced w.r.t. } $\Lambda$ in the following sense:  $K_0 = p^{-1}(p(K_0))$. We shall assume throughout that $K_0\subset K=p^{-1}(M)$ is such that $p(K_0)=M$.

Assuming  that $K$ is balanced and contains no critical points, then the function $f$ is pointwisely uniquely determined by the values of $\varphi$.  In order to write down a formula we agree about some additional notation.
 Denote the roots of $p(z)-w=0$ by $z_j=z_j(w)$. Away from critical values these are analytic and we assume a fixed numbering so that $z_j(w)\rightarrow \lambda_j$   if $z_1(w) \rightarrow \lambda_1$ (when  $w\rightarrow 0$). 
 In the inversion we  essentially exchange  interpolation and evaluation points.  To that end let $\delta_j(\zeta; w)$ denote the interpolation polynomial, with $w$ fixed, which takes the value 1 at $\zeta=z_j(w)$ while vanishing at other $z_k(w)$'s:
\begin{equation}\label{kaanteismuoto}
\delta_j(\zeta;w)= \frac{p(\zeta)-w}{p'(z_j(w))(\zeta-z_j(w))},
\end{equation}
so that in particular $\delta_j(\zeta;0)=\delta_j(\zeta).$

\begin{proposition}\label{kaanteispropo}Suppose $K$ is a balanced compact set with respect to local centers $\Lambda$.  Assume that $\varphi$ is given pointwisely in  $K$.  Then $f$ is uniquely defined for all noncritical values $w\in M \setminus p(\Lambda_1)$   by
 \begin{equation}\label{INV}
f_k(w)= \sum_{j=1}^d \delta_j(\lambda_k;w)  \varphi(z_j(w)).
\end{equation} 
The functions $f_k$ inherit the smoothness of $\varphi$, and additionally, if $\lambda_c \in \Lambda_1$  is an interior point of $K$ and $\varphi$ is at that point analytic, then the singularities  of each $f_k$  at the critical value $p(\lambda_c)$ are removable.  

\end{proposition}
\begin{proof}
See the discussions in [9] and [10].

\end{proof}
So, we could use the expression
$
f= \mathcal L^{-1} \varphi
$
at least when the components of $f$ are determined by (\ref{INV}) for noncritical values $w$ provided  $\varphi$ is given in a balanced set.  In particular this is natural when $\varphi$ is analytic in a balanced domain.  However, the topic of this paper is in functions which are perhaps given only on discrete sets, such as the set of eigenvalues of a matrix and then  some extra care is needed  in considering the possible lack of injectivity of  $\mathcal L$.  We shall therefore build a Banach algebra and view $\mathcal L$ as performing the Gelfand transformation $\hat f= \mathcal L f$. We then get many  general properties of Gelfand transform to be transported into our situation with relatively small amount of work.

 \subsection{Multiplication of  the vector functions:  polyproduct}

 Consider now continuous  functions $f$ mapping $M$ into $\mathbb C^d$.  We are aiming to define a Banach algebra structure into $C(M)^d$.  Denoting by $\circledcirc $ the multiplication in $C(M)^d$ we then want that $\mathcal L$ takes the vector functions into scalar functions in such a way that  
  $\mathcal L$ becomes an algebra homomorphism 
$$
 \mathcal L ( f \circledcirc g)= (\mathcal L f)( \mathcal L g)
$$
where the multiplication  of scalar functions $\mathcal L f$  is pointwise.

Since $\sum_{j=1}^d \delta_j(z) =1$ the  constant vector ${\bf 1}=(1,\dots,1)^t \in \mathbb C^d$ serves as the unit in the algebra.   In order to  define $ f \circledcirc g$ we hence need to code  the differences between components of $f$. 


\begin{definition}
For $a\in \mathbb C^d$ we set 
$$
\Box: \ \     a \mapsto { \Box a} =\left( \begin{matrix} 0& a_1-a_2&\dots & a_1-a_d\\
a_2-a_1& 0 & \dots & a_2-a_d\\
\dots&\dots& \dots&\dots \\
a_d-a_1 & \dots & a_d-a_{d-1}  & 0
\end{matrix}\right)
$$
and call it {\it boxing} the vector  $a$.
\end{definition}
In order to define the product we still need to introduce  two "scaling" entities, matrix $L$ and vector $\ell$.
 The matrix $L$ has zero diagonal  and 
$
L_{ij}=  {1} / ({\lambda_i - \lambda_j} )$  for $i\not=j$,
while the vector $\ell\in \mathbb C^d$  has components $\ell_j=1/p'(\lambda_j)$.
 Now, denoting by $\circ$ the Hadamard (or Schur,  elementwise) product
we can define the product  as follows.

\begin{definition}
Let $f$ and $g$ be pointwisely defined functions  from $M \subset \mathbb C$ into $\mathbb C^d$.  Then their "polyproduct"  $f\circledcirc g$ is a function defined on $M$, taking values in $\mathbb C^d$ such that
$$
(f\circledcirc g)(w) = (f\circ g)(w) - w\ ( L\circ {\Box {f(w)}}\circ {\Box{g(w)}} \ ) \ell.
$$

\end{definition}

 \begin{remark}
We shall write this in short, with slight abuse of notation, as
$$
f\circledcirc g = f\circ g - w\ ( L\circ \Box f\circ \Box{g} \ ) \ell.
$$
Further,  for the powers we write  $f^n = f\circledcirc f^{n-1}$ and the inverse in particular as $f^{-1}$ whenever it exists:  $f\circledcirc f^{-1}=\bf 1$.
\end{remark}

 \begin{proposition} The  vector space of functions 
$$
f \ : \ M\subset \mathbb C \rightarrow \mathbb C^d
$$
equipped  with the product $\circledcirc$ becomes a complex commutative algebra with ${\bf 1}$ as the unit.
\end{proposition}

\begin{proof}
In addition to the obvious   properties of scalar multiplication and summation of vectors we observe  that the  vector product is  commutative
$$
f \circledcirc g= g \circledcirc f
$$
and since $\Box {\bf 1}=0$ we have
$
{\bf 1}\circledcirc f =f.
$
Further, since $\Box( {\alpha f + \beta g})= \alpha \Box f + \beta \Box g,
$
we get 
$$
(\alpha f+ \beta g) \circledcirc h = \alpha (f\circledcirc h) + \beta (g \circledcirc h).
$$
These are enough for the structure to be an algebra.
\end{proof}

\begin{theorem}  Let $f$ and $g$ be defined in $M$ and $K= p^{-1}(M)$.  Then  if  $\varphi$ and $\psi$ are functions defined on $K$    by
$
\varphi = \mathcal L f $ and  $ \ \psi= \mathcal L g ,
$
then $\varphi \psi$  is given by
$$
\varphi \psi  = \mathcal L ( f \circledcirc g).
$$
\end{theorem} 

\begin{proof}

When we multilply $\varphi$ and $\psi$  products $\delta_i\delta_j$ appear.  For writing the expressions in a simple form we introduce
\begin{equation}
\sigma_{ij}= \frac{1}{p'(\lambda_j)} \frac{1}{\lambda_i-\lambda_j}.
\end{equation}

\begin{lemma}\label{tulonkonsistenssi}
We have with $w=p(z)$
\begin{equation}\label{samaindex}
\delta_i^2(z)= \delta_i(z) - w \sum_{j\not=i} [\sigma_{ij} \delta_i(z) +\sigma_{ji} \delta_j(z)]
\end{equation}
while for $i\not=j$
\begin{equation}\label{eriindex}
\delta_i(z)\delta_j (z)=  w\   [\sigma_{ij} \delta_i(z) +\sigma_{ji} \delta_j(z)].
\end{equation}
\end{lemma}
{\it Proof of the lemma.} 
Let first $i\not=j$. Since $$\delta_i = \frac{p}{p'(\lambda_i)(z-\lambda_i)}$$ and $p(z)=w$ we can write
$$
\delta_i \delta_j = \frac{w}{p'(\lambda_i)} \frac{\delta_j}{z-\lambda_i}.
$$
But
$
{\delta_j}/{(z-\lambda_i)}$ is a polynomial of degree $d-2$ and can thus be written as a linear combination in these basis polynomials. This gives
$$
\frac{\delta_j}{z-\lambda_i} = \frac{1}{\lambda_j -\lambda_i} \delta_j  + \frac{p'(\lambda_i)}{p'(\lambda_j)(\lambda_i-\lambda_j)} \delta_i
$$
which then yields (\ref{eriindex}).

Consider then (\ref{samaindex}). Since $\sum_j \delta_j =1$  we can write $\delta_i = 1-\sum_{j\not=i} \delta_j$ to get
$$
\delta_i^2 = \delta_i - \sum_{j\not=i} \delta_i \delta_j,
$$
which,  with the help of (\ref{eriindex}), yields the claim and completes the { proof of the lemma.}

 \smallskip

We can now multiply the expressions for $\varphi$ and $\psi$. 
\begin{align}
\varphi \psi&= \sum_{i,j} \delta_if_i \delta_jg_j\\
&= \sum_{i} \delta_i^2 f_i g_i + \sum_i \sum_{j\not=i} \delta_i\delta_j f_i g_j\\
&=\sum_i\delta_i f_ig_i  - w \sum_{i}\sum_{j\not=i}[\sigma_{ij} \delta_i(z) +\sigma_{ji} \delta_j(z)]f_ig_i\\
&+ w \sum_i\sum_{j\not=i} [\sigma_{ij} \delta_i(z) +\sigma_{ji} \delta_j(z)]
f_i g_j.
\end{align}
Here the term multiplying $\delta_k$  appears in the form
$$
\delta_k f_k g_k -w \sum_{j\not=k} \sigma_{kj} (f_k-f_j)(g_k-g_j)
$$
and hence the whole expression reads
$$
\varphi \psi = \sum_{i} \delta_i  \ [ f_ig_i - w \sum_{j\not=i}\sigma_{ij} (f_i-f_j)(g_i-g_j)].
$$
This is easily seen to be of  the form  $\varphi\psi= \mathcal L( f\circledcirc g)$ which completes the proof of the theorem.
 
\end{proof}

\subsection{The norm in the algebra}

We shall be considering continuous functions $f$  from a compact $M\subset \mathbb C$ into $\mathbb C^d$ and  begin with the uniform norm
$
|f|_M = \max_{w\in M}  |f(w)|_\infty
$
where $|a|_\infty= \max_{1\le j\le d}|a_j|$.
 The definition of  polyproduct  makes this into an algebra, but $|\cdot|_M$ is not an algebra norm in general,  so we need to move into the "operator norm".
 
 \begin{definition} 
 For $f\in C(M)^d$ we set
 $$
\| f\| = \sup_{|g|_M \le 1} | f \circledcirc g |_M.
$$
\end{definition}
This is clearly a norm in $C(M)^d$ and it is in fact equivalent with $|\cdot|_M$.

\begin{proposition}\label{epayhtaloitanormeille}
There is a constant $C$, only depending on $M$ and on $\Lambda$ such that
\begin{equation}\label{prodcont}
\|f \circledcirc g\| \le \|f\| \|g\|
\end{equation}
\begin{equation}
|f|_M \le \|f\|  \le C |f|_M.
\end{equation}
 
\end{proposition}

\begin{proof}
In fact
$$
|f|_M =|f \circledcirc {\bf 1} |_M \le \|f\|.
$$
On the other hand, from the definition of the polyproduct it is clear that there exists a constant $C$ such that
$$
 | f \circledcirc g|_M  \le C  |f|_M \  |g|_M.
 $$
But then 
$$
\| f\| = \sup_{|g|_M \le 1} | f \circledcirc g |_M \le C |f|_M.
$$
Finally, 
$$
|f\circledcirc g \circledcirc h|_M  \le \|f\| \  | g\circledcirc h|_M \le \|f\|?? \|g\| ? |h|_M
$$
implies (\ref{prodcont}).

\end{proof}
Since the polyproduct $\circledcirc$ is uniquely determined by $\Lambda$, we  shall denote the algebra in short as $C_\Lambda(M)$.

\begin{definition}
The vector space $C(M)^d$ of continuous functions $f$  from a compact $M\subset \mathbb C $ into $\mathbb C^d$,  with the operator norm $\|f\|$ and product $\circledcirc$ is denoted by
$C_\Lambda(M)$.
\end{definition}

The discussion can be summarized as follows.

\begin{theorem}  The Banach space $C(M)^d$ equipped with polyproduct $\circledcirc$, and denoted by 
 $C_\Lambda(M)$, is a commutative unital Banach algebra.  The algebra-norm $\| \cdot \| $  is equivalent with  $| \cdot |_M$ and functions with components given by polynomials $p(w,\overline w)$ are dense in $C_\Lambda(M)$.
\end{theorem}
\begin{proof}
Recall that polynomials $p(w,\overline w)$ are dense in the sup-norm on a compact  $M\subset \mathbb C$ among continuous functions by  Stone-Weierstrass theorem.  Applying this on each component of functions $f\in C_\Lambda(M)$ gives the result.
\end{proof}


\subsection{Characters of $C_\Lambda(M)$}

 In order to be able to apply the Gelfand theory  we need to know all characters in the algebra $C_\Lambda (M)$.
\begin{definition}
A   nontrivial linear functional $\chi: C_\Lambda(M) \rightarrow \mathbb C$ is  called a character if it is additionally multiplicative:
$$ \chi(f\circledcirc g)=\chi(f)\chi(g).$$  The set of all characters is the character space, which we denote here by $\mathcal X $.
\end{definition}

\begin{remark}
In commutative unital Banach algebras all characters - complex homomorphisms -  are automatically bounded and of norm 1.
Since  maximal ideals are kernels of characters, the focus is sometimes on the maximal ideals rather than on the characters,  [1], [2], [3], [13].

\end{remark}

Because the polyproduct $\circledcirc$ is constructed to yield 
$
\mathcal L(f\circledcirc g)= \mathcal L f \mathcal L g,
$
we conclude immediately that for each fixed $z_0\in p^{-1}(M)$ the functional
\begin{equation}\label{karakteeri}
\chi_{z_0} : f\mapsto \sum_{j=1}^d \delta_j(z_0) f_j(p(z_0))
\end{equation}
is a character.  We show next that there are no others.

\begin{theorem}\label{kaikkikarakteerit}
The character space $\mathcal X$ is  
$$
\mathcal X=\{\chi_z : z\in p^{-1}(M) \}
$$
where $\chi_z$ is given in  (\ref{karakteeri}).   

\end{theorem}
\begin{proof}
We need to show that all characters are of the form (\ref{karakteeri}).  Let $\chi\in \mathcal X$ be given and apply it within the subalgebra consisting of  elements of the form
$$
f= \alpha {\bf 1},
$$
where $\alpha$ is a  scalar function $\alpha \in C(M)$.  
Now,  it is well known that all  multiplicative functionals in $C(M)$ are given by evaluations at some $w_0\in M$:  $\alpha\mapsto \alpha(w_0)$;  hence $\chi(\alpha {\bf1})= \alpha(w_0)$ for some $w_0\in M$.

Next, take an arbitrary $g\in C_\Lambda(M)$.  Then  we conclude from
$$
\chi(\alpha {\bf 1} \circledcirc g) = \alpha(w_0) \chi(g)
$$
that $\chi(g)$ depends on $g(w_0)$, only. In fact $\chi(\alpha {\bf 1} \circledcirc g)=\chi(\alpha g)= \alpha(w_0) \chi(g)$ and  if $\alpha(w_0)=1$ we have
$$
\chi(g) = \chi(\alpha g)+ \chi((1-\alpha)g)
$$
so that $ \chi((1-\alpha)g)=0$. 

We assume next that $w_0$ is chosen and $\chi$ is a character $f\mapsto \chi(f)$ such that  the value only depends on $f(w_0)$.  We may therefore view $\chi$ as an  arbitrary linear functional in $\mathbb C^d$  which is multiplicative with respect to the polyproduct $\circledcirc$ at $w_0$.   In fact, setting  for $a,b\in \mathbb C^d$
$$
a b = (a\circledcirc b)(w_0)
$$
makes $\mathbb C^d$ into a Banach algebra, for each fixed $w_0$. 

Let   $a,b\in \mathbb C^d$, then $\chi$ is of the form
$$
\chi(a)=\sum_{j=1}^d \eta_j a_j
$$
and we require
$$
\chi((a\circledcirc b)(w_0))= \chi(a) \chi(b).
$$
First observe that $\chi({\bf 1})=1$ gives $\sum_{j=1}^d \eta_j=1$.  Then,  comparing with  Lemma \ref{tulonkonsistenssi} and using the notation in the proof of it,   we see that we must have
\begin{equation}\label{samaindeksi}
\eta_i^2= \eta_i - w_0 \sum_{j\not =i}(\sigma_{ij} \eta_i+\sigma_{ji} \eta_j)
\end{equation} 
while for $j\not=i$
\begin{equation}\label{eriindeksit}
\eta_i\eta_j = w_0 (\sigma_{ij} \eta_i+\sigma_{ji} \eta_j).
\end{equation}

Let first $w_0=0$.  We have then $\eta_i \in \{0,1\}$ and,  since $\sum_i\eta_i=1$, exactly one $\eta_j=1$. But for $w_0= p(\lambda_i)=0$ and thus $\eta_j=\delta_j(\lambda_j)=1$ and there are exactly $d$ different solutions.

For $w_0 \not=0$ we have from  (\ref{eriindeksit}) that $\eta_i\not=0$ for all $i$.  We take $\eta_1$ as an unknown  so that for $j>1$
$$
\eta_j= \frac{w_0 \sigma_{1j}\eta_1}{\eta_1-w_0 \sigma_{j1}}.
$$
Substituting these into (\ref{samaindeksi}) and  dividing with $\eta_1\not=0$ yields 
$$
\eta_1=1-w_0 \sum_{j\not=1} \sigma_{1j} -w_0^2 \sum_{j\not=1} \frac{\sigma_{j1}\sigma_{1j}}{\eta_1-w_0 \sigma_{j1}}.
$$
 This has, counting multiplicities, exactly $d$ solutions for $\eta_1$. However, we already know $d$ solutions, namely, $\delta_1(z_k(w_0))$, for $k=1, \cdots,d$ where $p(z_k(w_0))=w_0$, which completes the proof.
 
\end{proof}


\subsection{Gelfand transform and the spectrum}

When $\varphi$ is holomorphic it is natural to think $\varphi$ as the "primary" function which is represented or  parametrized by the vector function $f$.  However,  when dealing with functions with less smoothness it is  easier to think their roles to be reversed.  This is because $f$ can be taken as any continuous  vector function while  the behavior of $\varphi$ is in general complicated near critical points. 

We  take $C_\Lambda(M)$ as the {\it defining} algebra while the functions $\varphi$ appear as  {\it Gelfand transforms}.  

Before applying this machinery we recall some basic properties of  Gelfand theory.
Let $\mathcal A$ be a  commutative unital Banach algebra  with unit $e$ and  denote by $h$ a character:
$$
h(ab)=h(a) h(b)  \  \text{  for all  } \ a,b \in \mathcal A.
$$
Let us denote by $\Sigma_\mathcal A$ the character space of  $\mathcal A$.   Then every $h \in \Sigma_\mathcal A$ has norm 1 and  $\Sigma_\mathcal A$ is  compact in the {\it Gelfand topology}:  one gives $\Sigma_\mathcal A$ the relative weak$^*$-topology it has as a subset of the dual  of $\mathcal A$.

Then the {\it Gelfand transform }   of $a \in \mathcal A$ is 
$$
\hat a:  \Sigma_\mathcal A \rightarrow \mathbb C  \  \text{ where } \hat a (h)=h(a).
$$
The function $\hat a$ is then always continuous in the Gelfand topology and this allows one to study the algebra $\mathcal A$ by studying continuous functions on $\Sigma_\mathcal A$.

Since every maximal ideal of $\mathcal A$ is of the form $\mathcal N_h=\{ a\in \mathcal A : h(a)=0\}$, the character space  is sometimes called the maximal ideal space of $\mathcal A$.
We collect here basic facts on the Gelfand theory, and here we treat $\hat a\in C(\Sigma_\mathcal A)$ as a continuous function with the sup-norm.   Recall that the {\it spectrum}  $\sigma(a)$ of an element $a\in \mathcal A$ consists of those $\lambda \in \mathbb C$ for which $\lambda e-a$ does not have an inverse in $\mathcal A$.  We denote by $\rho(a)$ the spectral radius of $a$:  $\rho(a) = \max \{  |\lambda| \ : \lambda \in \sigma(a)\}$.

\begin{theorem}\label{GELFAND}(Gelfand representation theorem)
Let $\mathcal A$ be a commutative unital Banach algebra.  Then for all $a \in \mathcal A$
\begin{itemize} 
\item[(i)]   $\sigma(a) = \hat a(\Sigma_\mathcal A) = \{ \hat a(h) : \  h \in \Sigma_\mathcal A \}$;
\item[(ii)]   $\rho(a)= \|\hat a\|_\infty =\lim_{n\rightarrow \infty} \|a^n\|^{1/n}  \le \|a\|;$
\item[(iii)]    $a\in \mathcal A$ has an inverse if and only if $\hat a(h) \not=0$ for all $h \in \Sigma_\mathcal A$;
\item[(iv)]    {\rm rad} $\mathcal A = \{ a \in \mathcal A : \ \hat a(h)=0$  for all $h \in \Sigma_\mathcal A\}$.

\end{itemize}

\end{theorem}

(See any text book treating Banach algebras, e.g. [1], [2], [3],  [13]).


We shall now  consider $C_\Lambda(M)$.  In what follows we write $f\circledcirc f^{n-1}=f^n$ and in particular $f^{-1}$  for the inverse of $f$.
Recall that we denote by $\mathcal X$ the  character space of $C_\Lambda(M)$  
$$
\mathcal X=\{ \chi_z:  z\in p^{-1}(M)\}
$$
where
$$
\chi_z (f) = \sum_{j=1}^d \delta_j(z) f_j(p(z)).
$$
This allows us to identify $\chi_z$ with $z$ and consequently $\mathcal X$ with $p^{-1}(M)$. Hence we shall view the Gelfand transform $\hat f$ as a function of $z \in p^{-1}(M)$.

\begin{definition}
Given $f\in C_\Lambda(M)$   we set
 
$$
\hat f: p^{-1}(M) \rightarrow \mathbb C
$$
$$
\hat f: z \mapsto \hat f(z) = \sum_{j=1}^d \delta_j(z) f_j(p(z)).
$$

 \end{definition}
 Thus,  we  can view the multicentric representation operator $\mathcal L$   as  performing the Gelfand transformation
 $$
\mathcal L:   f \mapsto \hat f.
$$
We denote this Gelfand transformation by $\mathcal L$ to remind that  for constant vectors $a\in \mathbb C^d$  the transformation $\hat a$ is just the Lagrange interpolation polynomial (restricted into $p^{-1}(M)$).   We denote $|\hat f|_K=\sup_{z\in K} |\hat f(z)|$.

We  specify now the general Gelfand representation theorem  for the algebra  $C_\Lambda(M)$.

\begin{theorem}\label{peruslause}(Multicentric representation as Gelfand transform)

For $ f\in C_\Lambda (M)$  the following hold with $K=p^{-1}(M)$: 
\begin{itemize}
 \item[(i)]   $\sigma(f) =  \{ \hat f (z): \  z \in K \}$;
\item[(ii)]   $\rho(f)= |\hat f|_K =\lim_{n\rightarrow \infty} \|f^{n}\|^{1/n}  \le \|f\|;$
\item[(iii)]    $f$ has an inverse if and only if $\hat f(z) \not=0 \ $ for all $z\in K$;
\item[(iv)]    {\rm rad} $C_\Lambda(M)= \{ f \in C_\Lambda(M): \ \hat f(z)=0$  for all $z \in K\}$.

\end{itemize}

\end{theorem}

  Recall, that an algebra $\mathcal A$ is called {\it semi-simple} if rad $\mathcal A=\{0\}$.


\begin{theorem}\label{semisimppelilause}
$C_\Lambda(M)$ is semi-simple if and only if  $M$ contains no isolated critical values of $p$.\end{theorem}
\begin{proof}
Take $f\in {\rm rad}(C_\Lambda(M))$ so that  $\hat f(z)=0$ for all $z\in K$.  Since $\hat f$ determines $f$ uniquely outside critical values,  we have $f(w)=0$ away from the critical values.  If every critical value is an accumulation point of $M$ then by continuity $f$ vanishes everywhere.  On the other hand, if $w_0\in M$ is an isolated critical value, take  critical points $z_i\in p^{-1}(\{w_0\})$.  By assumption, they are isolated and  all we need to do is to find $0\not= a\in \mathbb C^d$ such that  $\sum_j \delta_j(z_i) a_j=0$ for all $i$. However,  since $w_0$ is a critical value at least two of the roots $z_i$ coincide   and hence the matrix $(\delta_j(z_i))_{ij}$ is not of full rank.  So, we conclude that nontrivial  solutions $f$ exist and $C_\Lambda(M) $ is not semi-simple.

\end{proof}
 
\begin{remark}
If $s_A$ is a simplifying polynomial of minimal degree for an $n \times n$ matrix $A$ (see Definition \ref{simply}), then {\it all} critical values of $s_A$ are isolated and inside $\sigma(s_A(A))$.
\end{remark}

\subsection{Invertible elements  of $C_\Lambda (M)$}
From Theorem \ref{peruslause} conclude that if $\varphi$ is given by multicentric representation 
$\varphi=\mathcal L f$ where $f$ is continuous and bounded, then $1/\varphi= \mathcal L g$ with a bounded and continuous $g$ if and only if $\varphi(z)\not=0$ for $z\in p^{-1}(M)$.
We  shall now derive a quantitative version of this.   


\begin{theorem}\label{invertible}   There exists a constant $C$  depending on $M$ and $\Lambda$ such that the following holds.   If $f\in C_\Lambda(M)$ is  such that for all $z \in p^{-1}(M)$  
$$
|\mathcal L f(z)| \ge \eta > 0,
$$
then there exists $g\in C_\Lambda(M)$ such that $f\circledcirc g = {\bf 1} $ and
\begin{equation}\label{estimalhaalta}
\|g\| \le C \  \frac{\|f\|^{d-1}}{\eta^d}.
\end{equation}
\end{theorem}

 Before  turning to prove this we look at an instructive example.
\begin{example}
We shall first consider the degree two case with $w=z^2-1$.  If we put $\varphi (z) = \mathcal L f(z)$, then the inverse $g=f^{-1}$ is given simply as follows:
\begin{equation}\label{2inv}
\begin{pmatrix}
g_1(w)\\
g_2(w)
\end{pmatrix}
= \frac{1}{\varphi(z)\varphi(-z)} \begin{pmatrix}
f_2(w)\\
f_1(w)
\end{pmatrix}.
\end{equation}
In fact, since
$$
g\circledcirc f=g\circ f + \frac{w}{4} (g_1-g_2)(f_1-f_2) {\bf {1}} 
$$
we have 
$$
g\circledcirc f = \frac{1}{\varphi(z) \varphi(-z)} (f_1f_2 -\frac{w}{4}(f_1-f_2)^2) \bf 1 
$$
and then expanding   
$\varphi(z)\varphi(-z) $
we obtain
$$ [\frac{1+z}{2} f_1+\frac{1-z}{2} f_2]
[\frac{1-z}{2} f_1+\frac{1+z}{2} f_2]=
f_1f_2 -  \frac{w}{4} (f_1-f_2)^2.$$
In particular, the constant $C$ in (\ref{estimalhaalta})  equals $ 1$ in this case.

\end{example}

\begin{proof} 
The example suggests  to look at the inverse in the following  way.
Denote  by $z_j=z_j(w)$ the roots of $p(\zeta)-w=0$ and when needed, we put $z_1(p(z))=z$.  With   $\varphi = \mathcal L f$ and  $\psi=1/\varphi = \mathcal L g$ we then have
$$
\psi(z_1) = \frac{1}{\Phi(w)}{\prod_{j=2}^d \varphi(z_j)} 
$$
where  $\Phi(w)= {\prod_{j=1}^d \varphi(z_j)}$.  We need the following lemma.


\begin{lemma}\label{riippuuvaintuplaveesta}
Suppose that $f$  has analytic components in $M$.  Then 
\begin{equation}\label{isofii}
\Phi: \ w \mapsto {\prod_{j=1}^d \varphi(z_j(w))}
\end{equation}
is analytic in $M$.   
\end{lemma}

{\it Proof of lemma.}
All roots $z_j(w)$ are analytic except possibly at critical values.  Since $\varphi(z_j(w))$ is given  by
$$
\varphi(z_j(w)) = \sum_{k=1}^d \delta_k(z_j(w)) f_k(w)
$$
with $f_k$'s analytic, we may as well assume that  $f_k$'s are constants as the only source for lack of analyticity at the critical values would come from products of $\delta_k$'s. But if    $f_k$'s  are constants, we may put   $q(\zeta)=\sum_{k=1}^d f_k \delta_k(\zeta)$. However,  then 
$$
p(\zeta_1, \cdots,\zeta_d) = \prod_{j=1}^d q(\zeta_j)
$$
is a symmetric polynomial and it can be expressed uniquely by elementary polynomials $s_i$ by Newton's theorem. 
If we now substitute $\zeta_j=z_j(w)$,   where $z_j(w)$'s are the roots of $p(\zeta)-w=0$, we observe that all elementary polynomials $s_i$ except $s_d$ are constants.  For example, $s_1= -\sum_{j=1}^d z_j(w)= -\sum_{j=1}^d \lambda_j$, while $s_d(w)= (-1)^d (p(0)-w)$.  Thus we arrive at a polynomial in $w$, which completes the proof of the lemma.

Next we  concentrate on   
$$
{\Phi(w)}\psi(z) =   \prod_{j=2}^d \sum_{k=1}^d \delta_k(z_j)f_k(w)
$$
and organize this as a sum of the form
$$
\delta_1(z_2)\cdots \delta_1(z_d) \ f_1(w)^{d-1}  + \  \cdots = \sum_{|\alpha|=d-1}q_{\alpha}(z)F_\alpha(w),
$$
where $\alpha=(\alpha_1, \cdots, \alpha_d)$ and $F_\alpha(w) = \prod_{k=1}^d f_k(w)^{\alpha_k}$ while $q_\alpha(z) $ is a rather complicated sum of products of different $\delta_k$'s evaluated at $z_j$'s with $j>1$.  Treating  $z_j=z_j(p(z))$ as functions of $z$, 
 $q_\alpha(z)$ are clearly analytic away from the critical points $z\in \Lambda_1$. Individual products of $\delta_k(z_j)$'s within $q_\alpha(z)$ may have branch points at these critical points while the sum $q_\alpha(z)$ itself is however a polynomial.  To see this, let $f_k(w)=x_k$ be constants and denote $x=(x_1, \cdots,x_d)^t \in \mathbb C^d$.  Then $F_\alpha(w)=x^\alpha$ and  if we put
$$
P(z,x)=  \sum_{|\alpha|=d-1}q_{\alpha}(z)x^\alpha,
$$
then we can view $P(z,x)$ as a   polynomial in $\mathbb C \times\mathbb C^{d}$.  In fact,  $\Phi(p(z))$ is a polynomial and $\varphi(z)$ divides it so $P(z,x)$ must be a polynomial in $z$.  But then, for example by  differentiating $P(z,x)$ with $\partial ^\alpha =\prod (\frac{\partial}{\partial x_k})^{\alpha_k}$ gives $\partial^\alpha P(z,x) = \alpha! q_\alpha(z)$ showing that each $q_\alpha$  is a polynomial in $z$.

Finally write $q_\alpha(z) = \sum_{j=1}^d \delta_j(z)Q_{\alpha,j}(w)$ so that   
$$
\Phi(w) \psi(z) = \sum_{j=1}^d \delta_j(z)   \sum_{|\alpha|=d-1} F_\alpha(w)Q_{\alpha,j}(w) 
$$ 
and, written in $C_\Lambda(M)$,
$
g= \sum_{|\alpha|=d-1} F_\alpha Q_{\alpha} / \Phi.  
$
Since $|\Phi(w)| \ge \eta^d$ and $|F_\alpha(w)| \le \|f\|^{d-1}$  the claim follows with
$C=\sum_{|\alpha|=d-1} \|Q_\alpha\|$.
 
\end{proof}


\subsection{ Characteristic function, resolvent estimates and nilpotent elements}

From the previous discussion we see that $f$ is invertible in the algebra if and only if $\Phi$  does not vanish.   This suggests   to introduce a characteristic function for $f$. This gives still another view to the algebra.

 Denoting  again by $z_j(w)$ the roots
 of  $p(z)-w=0$ we have  $\lambda \in \sigma(f)=\{\hat f(z) \ : z \in K=p^{-1}(M) \}$ if and only if $\prod_{j=1}^d (\lambda-\hat f(z_j(w)))=0$  at some $w\in M$. 
Expanding the product as a polynomial in $\lambda$   takes the form
\begin{equation}\label{maaritteleekarakteristin}
\pi_f(\lambda,w)=\prod_{j=1}^d (\lambda-\hat f(z_j(w)))= \lambda^d - \Phi_1(w) \lambda^{d-1} + \cdots + (-1) ^d \Phi_d(w),
\end{equation}
since the coefficient functions  $\Phi_j$ are again  functions of $w$ by   the same argument as in Lemma \ref{riippuuvaintuplaveesta}; notice that  $\Phi_d$ equals the $\Phi$ in (\ref{isofii}). 

\begin{definition}\label{karakteristinenfunktio}
Given $f \in C_\Lambda(M)$ we call $\pi_f(\lambda,w)$ the characteristic function of $f$.
Further, we call $(\lambda {\bf 1} -f)^{-1}$ the resolvent element whenever it exists.

\end{definition}

This allows us to formulate a different version of  the estimate for the inversion.
To that end we denote by $|\cdot|_\infty$ the max-norm in $\mathbb C^d$.
 
 \begin{theorem}
 There exists a constant $C$, depending on $M$ and on $\Lambda$, such that  for $w\in M$
 $$
 |(\lambda {\bf 1}-f)^{-1}(w)|_\infty \le \ C \  \frac{(|\lambda| + \|f\|)^{d-1}}{|\pi_f(\lambda,w)|}.
 $$
 
 \end{theorem}
 \begin{proof}
 This follows in an obvious way from  Theorem \ref{invertible}  and from the definitions.
 \end{proof}

\begin{remark}
From $\rho(f) = | \hat f|_K  \le \|f\|$ we have the lower bound
\begin{equation}\label{alarajaestimaatti}
\frac{1}{{\rm dist}(\lambda, \sigma(f))} =| \frac{1}{\lambda -\hat f}|_K \le \| (\lambda{\bf 1}-f)^{-1}\|.
\end{equation}
This in particular implies that if $f\not= \lambda {\bf 1}$ and $\lambda \in \partial \sigma(f)$,   there exists $g_n \in C_\Lambda(M)$ of unit length such that $(\lambda {\bf 1} -f ) \circledcirc g_n \rightarrow 0$.  In other words,   $\lambda {\bf 1}-f$ is a {\it topological divisor of zero}.
\end{remark}

We noted earlier that $f\in $ rad $C_\Lambda(M)$ if and only if $\hat f$ vanishes identically, or , which is the same thing, $\sigma(f)=\{0\}.$  

\begin{proposition} 
If  $\sigma(f)=\{0\}$, then $f$ is nilpotent and there exists $n \le d$ such that $f^n=0$.
\end{proposition}

\begin{proof}
In other words, we need to show that all quasinilpotent elements are actually nilpotent. It is clear from Theorem \ref{semisimppelilause} that nontrivial quasinilpotent elements exist when $M$ contains an isolated critical value, say $w_0$.  We can proceed now as follows. We view   the multiplication 
 $$
 f: g \mapsto f\circledcirc g
 $$
 as an operator in $C(M,\mathbb C^d)$
and hence  for each $w\in M$ there is a matrix $B_f(w)$ such that
$$
(f \circledcirc g) (w)=B_f(w) g(w).
$$
If $f$ is quasinilpotent, it means that each $B_f(w)$ must be for fixed $w$ quasinilpotent.  However, a $d\times d$-matrix is quasinilpotent  only when it is nilpotent, from which the claim follows.
\end{proof}


 \begin{example}
Consider  $w=z^2-1$.  Put $h=\frac{w}{4} (f_1-f_2)$.  Then
$$
B_f= \begin{pmatrix} f_1+h & -h\\
h & f_2-h
\end{pmatrix}
$$
which  for fixed $w$ has the eigenvalues 
$$
\frac{1}{2}(f_1(w) + f_2(w)) \pm \frac{\sqrt{1+w}}{2} (f_1(w)-f_2(w)).
$$
That is, the eigenvalues are simply $\varphi(z)$ and $\varphi(-z)$.
Denote 
$$
E(z)= \begin{pmatrix}
\delta_1(z)&\delta_2(z)\\
\delta_1(-z)& \delta_2(-z)
\end{pmatrix}
$$
so that  for $z\not=0$
$$
E(z)^{-1} = \frac{1}{2z}\begin{pmatrix}
z+1&z-1\\
z-1&z+1
\end{pmatrix}.
$$
Finally,
$$
\begin{pmatrix}
\varphi(z)&0\\
0&\varphi(-z)
\end{pmatrix} = E(z) B_f(w) E(z)^{-1}
$$
and we see that the eigenvectors are independent of the function $f$. At $z=0$ the eigenvalues agree,  and $E(0)$ is no longer invertible.
Put $f(-1)=(1, -1)^t$ so that $\varphi(0)=0.$  Then
$$
B_f(-1)=\frac{1}{2} \begin{pmatrix} 1&1\\
-1&-1\end{pmatrix}   \text { is similar to } \begin{pmatrix} 0&1\\ 0&0\end{pmatrix}.
$$
\end{example}


\subsection{Quotient algebra  $C_\Lambda(M)/\mathcal I_{K_0}$}

When we apply the functional calculus,  discussed in the next section, the natural requirement for $\varphi$ is  that it is well defined at the spectrum $\sigma(A)$ of the operator $A$, which means that  $f$  representing $\varphi$ must be well defined on a set which includes $p(\sigma(A))$.  However,    $p^{-1}(p(\sigma(A))$   is likely to be properly larger than $\sigma(A)$ which in practice shows up in lack of uniqueness in representing $\varphi$.  

Let $K_0\subset \mathbb C$ be compact,  put $p(K_0)=M$ and denote as before
$K=p^{-1}(M)$.  We assume here that the inclusion $K_0 \subset K $ is proper.

Let $\mathcal I_{K_0}$ be the  closed ideal in $C_\Lambda(M)$
$$
\mathcal I_{K_0} =\{ f \in C_\Lambda(M) : \hat f(z)=0 ?\ \text{ for } \ z \in K_0\}.
$$
Then the set of elements we are dealing with can be identified with the cosets $[f]$ :
$$
C_\Lambda(M) / \mathcal I_{K_0} =\{ [f]  : [f]= f+\mathcal I_{K_0}\}.
$$
This is a unital  Banach algebra with norm defined as
$$
\|?\ [f]  \ \|= \inf_{g\in \mathcal I_{K_0}} \|f+g\|.
$$
We need to identify the character space of this quotient algebra.
\begin{definition} 
Given a closed ideal $J \subset \mathcal A$ the hull of the ideal is  the set of all characters which vanish at every element in the ideal.
\end{definition}

\begin{lemma} (Theorem 6.2 in [5])
Given a closed ideal $J$ in a commutative Banach algebra $ \mathcal A$, the character space of  the quotient algebra $ \mathcal A/ J $ is the hull of $  J$.
\end{lemma}

\begin{corollary}\label{tahanviitataan}
The quotient algebra $C_\Lambda(M)/ \mathcal I_{K_0}$ is a Banach algebra with unit and the character space can be identified with $K_0$,  so that the Gelfand transformation becomes
$
  [f] \mapsto \hat f_{| K_0}.
$
\end{corollary}


\subsection{Additional remarks on $\mathcal L C_\Lambda(M)$}

Here we make some observations on the  range of  the Gelfand transformation.   Denoting  $\mathcal L C_\Lambda(M)
= \{\varphi\in C(K) : \exists f\in C_\Lambda(M)  \text { such that } \varphi=\hat f\}$ we clearly have a  normed subalgebra of $C(K)$ with the  sup-norm on $K=p^{-1}(M)$ but the algebra need not be closed.

\begin{example}   
Let $\Lambda=\{-1,1\}$ so that  $p(x)=x^2-1$, and  $M=[ -1,0]=\{x :  -1 \le x \le 0 \}$ so that $K=p^{-1}(M)=[-1,1] $.  Then  $\mathcal L C_\Lambda(M)$ contains all polynomials as any polynomial $Q(x)$  can uniquely be written as
$$Q(x)= \sum_{j=1}^d  \delta_j(x) Q_j(p(x))
$$
where $Q_j$'s are polynomials.  Now, polynomials are dense in $C(K)$ and we conclude that the closure of $\mathcal LC_\Lambda(M)$  equals $ C(K)$ in this case.  However,  if we take $\varphi \in C(K)$ such that 
$$
\varphi(x) =  \max\{x^\alpha 
,0\}
$$  
then for $0<\alpha<1$ we have $\varphi \in C(K)  \setminus \mathcal L C_\Lambda(M)$.  In fact,  for $x\not=0$ we have $\varphi (x) = \mathcal L f(x)$ with $f(x^2-1)$ becoming unbounded as $x$ tends to $0$.
Note, that in this example the Gelfand transformation is injective.

\end{example}
\begin{example}
Let $\Lambda =\{-1,1\}$ but $K=\Lambda_1=\{0\}$.  Then $C_\Lambda (\{-1\})$ is a two-dimensional complex algebra, with nontrivial radical consisting of  vectors $f$ such that
$f_1(-1)+f_2(-1)=0$.   On the other hand $ \mathcal L C_\Lambda(\{-1\})$ is one-dimensional, closed  and   isomorphic with the complex field.

\end{example}
  
\begin{example}\label{reikaleipa}  Let $\Lambda=\{-1,1\}$ and $K=\{ z\ : \varepsilon \le |z| \le 2\}$  with some  small positive $\varepsilon$.   Then the critical point, the origin, is not in $K$ and  the  following hold with some constant $C$
$$
\|\mathcal L f\|_\infty \le \|f \|  \le C \  \|\mathcal Lf \|_\infty.
$$
  However,  if   $\mathcal L f(x+iy)= x^\alpha$ with $0<\alpha<1$ for $x>0$ and vanishing on the left half plane , then  
$$
\| f\|  \sim \rm Const/ \varepsilon ^{1-\alpha}.
$$
 \end{example}

It is natural to ask whether $\varphi \in \mathcal LC_\Lambda(M)$ shall be differentiable at   the interior critical points. 
After all, we shall be able to apply the functional calculus   in such a case for matrices which do have a nontrivial Jordan block  and we usually assume that the value on the off-diagonal would be the derivative of $\varphi$ at the eigenvalue in question.

\begin{example}\label{perustaa}
Let again $p(z)=z^2-1$ but $K$ such that it contains the critical point  in the interior:  $K=\{z : |p(z)|\le 2\}$.  Then $M$ likewise contains a neighborhood of -$1$.   We have
$$
\varphi(z) = \frac{1}{2}[f_1(w)+f_2(w)] + \frac{z}{2}[f_1(w)-f_2(w)].
$$
If $\varphi \in \mathcal L C_\Lambda(M)$, then $f_i \in C(M)$ and we have
$$
\frac{1}{2z} [\varphi(z)-\varphi(-z)] = \frac{1}{2}[f_1(z^2-1)-f_2(z^2-1)] 
$$ and hence the limit  
$$
\lim_{z\rightarrow 0}\frac{1}{2z} [\varphi(z)-\varphi(-z)] = \frac{1}{2}[f_1(-1)-f_2(-1)] 
$$
always exists. However,  it does not imply that $\varphi$  would be  differentiable at the origin. In fact, we have
\begin{align}
&\frac{1}{z}[\varphi(z)-\varphi(0)] \\
=&\frac{1}{2z} \{ [f_1(z^2-1)+f_2(z^2-1)] -[f_1(-1)+f_2(-1)]\\
+&\frac{1}{2}[f_1(z^2-1)-f_2(z^2-1)].
\end{align}Here the last term is continuos as $z$ tends to origin. Thus the derivative exists depending on the  behavior of $f_1+f_2$ near $w=-1$.  In particular, if  $f_1+f_2$ is H\"older  continuous with exponent >1/2, then  $\varphi$ is differentiable.

\end{example}


\section{Functional calculi}

\subsection{Functional calculus for matrices}

We discuss first the functional calculus related to $C_\Lambda(M)$ for matrices. 
Denote by $\mathbb M_n$ complex $n \times n$-matrices with the norm 
$$
\| A\| = \sup_{|x|_2=1}|Ax |_2.
$$  
   Further, we  denote by $\sigma(A)=\{\alpha_k\}$ the eigenvalues of $A$
and by $m_A $ the minimal polynomial of $A$, that is,  the monic polynomial $q$ of smallest degree such that $q(A)=0 $:
$$
m_A(z)=\prod_{k=1}^m (z-\alpha_k)^{n_k+1}.
$$

As mentioned in the introduction, the usual way to formulate the class of functions $\varphi$  for which $\varphi(A)$ is well defined, asks the following to be known  at every eigenvalue $\alpha_k$
$$
\varphi(\alpha_k), \cdots, \varphi^{(n_k)} (\alpha_k),
$$
[6], [11].   Based on this information one can then construct an Hermite  interpolation polynomial $p$ and set $\varphi(A)=p(A)$.

 As we saw in Example  \ref{perustaa}   the functions in our algebra do not need to be differentiable - but of course when they are the resulting functional calculus yields the same matrices $\varphi(A)$.

\begin{definition}\label{simply}Given $A\in \mathbb M_n$  we call all monic polynomials $p$ such that $p(A)$ is similar to a diagonal matrix as simplifying polynomials for $A$.
\end{definition}

If $\mathcal K$ denotes those indices $k$ for which $n_k>0$ in the minimal polynomial, then  setting
$$
s_A(z)= \int_{0}^z \prod_{k\in \mathcal K}(\zeta-\alpha_k)^{n_k} d\zeta + c
$$
we have a polynomial of minimal degree such that $s_A^{(j)}(\alpha_k)=0$ for $j=1, \cdots, n_k$ and $k\in \mathcal K$.  Clearly then $s_A(A)$ is similar to the diagonal matrix ${\rm diag} (s_A(\alpha_k))$.  Since we can add an arbitrary constant to $s_A$ we may assume as well that $s_A$ has distinct roots.  

Let now $p$ be a simplifying polynomial for $A$ with distinct roots and assume $\varphi$ is given  on $\sigma(A)$ as
$$
\varphi(z)= \sum_{j=1}^d \delta_j(z) f_j(p(z)).
$$
Denoting  $B=p(A)$ we  could then  define for $f_j \in C(\sigma(B))$ the matrix function $f_j(B)$ either by Lagrange interpolation at $p(\alpha_k)$ or  by assuming the similarity transformation to  the diagonal form $B=T DT^{-1}$  be given and setting $f_j(B)=T f_j(D)T^{-1}$, both yielding the same matrix $f_j(B)$ which commute with $A$.  Then  the following matrix is well defined:  
$$
\varphi(A) = \sum_{j=1}^d \delta_j(A) f_j(B).
$$
It follows immediately that if we have two functions $f, g \in C_\Lambda(\sigma(B))$, and we denote $\varphi=\mathcal L f$, $\psi= \mathcal L g$ and $\varphi \psi = \mathcal L (f \circledcirc  g)$, then this definition yields
$$
 (\varphi\psi)(A) =\varphi(A)  \psi(A).
$$   
 However, we formulate the   exact statement  using a different notation to underline the fact that knowing the values of $\varphi$   at the spectrum of $A$ need not  determine $f$  uniquely,  and hence not $\varphi(A) $, either. 
 
\begin{definition} Assume $p$ is a simplifyng polynomial for $A \in \mathbb M_n$ with disting roots $\Lambda$. Then we denote by $\chi_A$  the mapping $C_\Lambda(p(\sigma(A))) \rightarrow \mathbb M_n$ given by
\begin{equation}
f \  \mapsto \ \chi_A(f)=  \sum_{j=1}^d \delta_j(A) f_j(B).
\end{equation}

\end{definition}
\begin{theorem}
The mapping $\chi_A$ is a continuous homomorphism $C_\Lambda(p(\sigma(A))) \rightarrow \mathbb M_n.$
\end{theorem}
 
 \begin{proof}
 That $\chi_A$ is a homomorphism  is build in the construction and in particular we have
 $$
 \chi_A(f\circledcirc g)= \chi_A(f)  \ \chi_A(g).
 $$
The continuity of $\chi_A$ is seen from
$$
\|\chi_A(f)\| \le \sum_{j=1}^d \| \delta_j(A) \|  \   \  \| f_j(B)\|
$$  combined with
$
\| f_j(B) \| \le \varkappa(T) |f|_{\sigma(p(A))}
$
and with $|f|_{\sigma(p(A))} \le ||f||$, see Proposition \ref{epayhtaloitanormeille}.  Here $\varkappa(T) =  \|T\|  \ \|T^{-1}\|$ denotes the condition number   of the  diagonalizing similarity transformation.

\end{proof}
We can now also conclude that  we can formulate a spectral mapping theorem.
Let $M=p(\sigma(A))$ and, as it would likely to be the case,  $\sigma(A)$ is a proper subset of $p^{-1}(M)$.  Then it follows from Corollary \ref{tahanviitataan} that the spectrum of $[f] $ in $C_\Lambda(M) /\mathcal I_{\sigma(A)}$ is
$
\sigma([f])= \{ \hat f(z) \ :  \ z \in \sigma(A)\}.
$

\begin{theorem}
We have for $[f] \in C_\Lambda(p(\sigma(A))) /\mathcal I_{\sigma(A)}$ and $\chi_A(f) \in \mathbb M_n$
$$
\sigma(\chi_A(f))  =  \sigma([f]).
$$
\end{theorem}

\begin{proof}

Even so the statement may look rather complicated  the proof here can  be reduced to the standard spectral mapping theorem for polynomials.  However, the statement holds as such  in more general setting and then  in particular the present simple proof is not available.

Consider $f_j(B)$ where $B=p(A)$ and denote by  $\beta_i$ the eigenvalues of $B$.  There are in general $s \le m$ different eigenvalues of $B$.   Let $q_j$  be the polynomial of degree $s-1$ such that 
\begin{equation}\label{lagra}
q_j (\beta_i)=f_j(\beta_i) \ \text { for } \ i=1, \cdots, s.
\end{equation}
Then we set $f_j(B)= q_j(B)$.   Thus we have
\begin{equation}\label{valmistulos}
\chi_A(f) = P(A)
\end{equation}
if we set $P(z)= \sum_{j=1}^d \delta_j(z)q_j(p(z))$.  The conclusion follows as $P$ is a polynomial.

\end{proof}

\begin{remark}
   There are two different steps to be taken when consructing $\chi_A(f)$.  

(i)    Given $A \in \mathbb M_n$ one could for example compute the Schur decomposition of $A$.  From there one must decide what diagonal elements are to be considered as the same and based on that one  chooses a  simplifying polynomial $p$  such that it has simple roots.  Notice in particular that then the eigenvalues $\alpha_k$ for which $n_k>0$, are distinct from the roots $\lambda_j$ of $p$.

(ii)  Given $f$ one then computes the Lagrange interpolating polynomials $q_j(w)$ satisfying  (\ref{lagra}) for each $j$.   

Then $\chi_A(f)$ is given by (\ref{valmistulos}).

\end{remark}
\begin{remark}
It is natural to ask how this approach is different from the  definition based on Hermite interpolation on the spectrum of $A$.  Consider the minimal polynomial $m_A$ as simplifying polynomial.  In the Hermite interpolation one interpolates at   the eigenvalues  while we   add a constant $c$  so that he polynomial $p(z)=m_A(z)+c$  has simple roots.   The effect on the differentiability requirement  on $f$  and/or $\varphi$   is then removed and replaced by a balanced limiting behavior of the roots of  $p$ near its critical points - and this  happens automatically, independent of the function $f$ as long as it is continuous.  To illustrate this, suppose $f$ is holomorphic and $\varphi = \mathcal L f$ so that 
$$
\varphi'(z)= \sum_{j=1}^d [\delta_j'(z) f_j(p(z)) + \delta_j(z) f_j'(p(z)) p'(z)] .
$$
However, at critical points $z_c$ we have $\varphi'(z_c)= \sum_{j=1}^d \delta_j'(z_c) f_j(p(z_c))$ so this value  does not depend on whether $f$ is differentiable at critical values  or not.  See also Example  \ref{perustaa}.

\end{remark}


\subsection{Polynomially normal operators in   Hilbert spaces}

We shall now consider bounded operators $A$  in complex Hilbert spaces $H$. The  operator norm  of $A\in \mathcal B(H)$   is denoted by $\| A \|$.
 
\begin{definition}   We call $A\in \mathcal B(H)$ polynomially normal, if there exists a nonconstant monic polynomial $p$ such that $p(A)$ is normal.  The polynomial $p$ is then called a simplifying polynomial for $A$. \end{definition}

Polynomially normal operators have been discussed in [4], [7], as operator valued roots for polynomial equations $p(z)-N=0$ with $N$ normal.   We formulate a  structure result (see Theorem 3.1, in [7], also  Theorem 2 in [8] ). 

\begin{theorem}\label{rakenne}
Let $H$ be separable and   $A \in \mathcal B(H)$  such that $p(A)$ is normal for some nonconstant polynomial $p$.  Then there exist  reducing subspaces $\{H_n\}_{n=0}^\infty$ for $A$, such that
 $
 H=\oplus_{n=0}^\infty H_n
 $
 and $
  A_{| H_0}$  is  algebraic while   
 $ A_{| H_n}$ are  for $n\ge1$   similar to  normal.
 \end{theorem}

We could take use of this structure result   but proceed independently of it.   We start by assuming that $p(A)$ is normal and then comment the  straightforward extension to the case where $p(A)$  is similar to normal.

Let $N=p(A)$ be normal, and as before,  we  may assume that $p$ has simple roots.   Then the first task is to define $f_j(N)$ in a consistent way. 
Recall  the following two results, see e.g.  [2].
\begin{lemma}
Let $M\subset \mathbb C$ be compact.  Then the closure of polynomials of the form
$q(w, \overline w)$
in the uniform norm over $M$ equals $C(M)$.
\end{lemma}
Since $N$ commutes with $N^*$ the operator $q(N,N^*)$ is well defined and the following holds.

\begin{lemma}
If $N\in \mathcal B(H)$ is normal, then
$$
\| q(N,N^*)?\| = \max_{w \in \sigma(N)} | q(w,\overline w)|.
$$
\end{lemma}

Given now a normal operator $N$ and a continuous function $f_j$ on $\sigma(N)$ one approximates  $f_j$ by a sequence $\{q_{j,n}\}$ such that
$$
| f_j -q_{j,n} |_\infty = \max_{w\in \sigma(N)} |f_j(w) - q_{j,n}(w, \overline w)| \rightarrow 0
$$
and sets
\begin{equation}
f_j(N)= \lim_{n\rightarrow \infty} q_{j,n} (N, N^*).
\end{equation}
Then $f_j(N) \in \mathcal B(H)$ is normal,  with  $\| f_j(N)?\| = |f_j|_\infty \le \|f\|.$

\begin{definition}\label{operaattoreillemaarit}
Assume $p$ is a simplifying polynomial for  $A\in \mathcal B(H)$  with distinct roots $\Lambda$,  so that $N=p(A)$ is normal.   Then we denote by $\chi_A$ the mapping $C_\Lambda(p(\sigma(A))) \rightarrow \mathcal B(H)$ given by
\begin{equation}
f \mapsto \chi_A(f)=\sum_{j=1} ^d \delta_j(A) f_j(N).
\end{equation}
\end{definition}
Note that  here $\delta_j(A)$ and $f_j(N)$ commute.  In fact,
$A$ commutes with $N=p(A)$ and since $N$ commutes with $N^*$  the operator $A$ commutes with $N^*$ as well, by Fuglede's theorem,  [2]. We  combine  the construction into  the following theorem.

\begin{theorem}
Let $A\in \mathcal B(H)$ and a simplifying polynomial $p$ be given as in Definition \ref{operaattoreillemaarit}.  Then the mapping $\chi_A$ is a continuous homomorphism from $C_\Lambda(p(\sigma(A)))$ to $\mathcal B(H)$.  In particular,
$$
\chi_A(f\circledcirc g)= \chi_A(f) \ \chi_A(g)
$$
and 
$$
\| \chi_A(f) \| \le   C   \| f\|  \  \text{ with }  \ C=\sum_{j=1}^d \| \delta_j(A)\|.
$$

\end{theorem}

\begin{remark} {\bf The case of $p(A)$ similar to normal.}
We can  extend the construction above to operators  which are similar to polynomially normal ones.  In short,   assume that $A\in \mathcal B(H)$ is such that there exists a  polynomial $p$ and a bounded $T$ with bounded inverse,  such that
$
N= T^{-1} p(A) T
$ is normal.  Denote $V=T^{-1} AT$ so that $N=p(V)$ and $B=p(A)$.  Then we can define
$$
f_j(B) = Tf_j(N) T^{-1}
$$
and again $A$ commutes with $f_j(B)$ as $A f_j(B)=T [V f_j(p(V))]T^{-1}$.   This allows us to define
\begin{equation}\label{similaarimaaritelma}
\chi_A (f)= T \chi_V (f) T^{-1}
\end{equation}
and the extension shares all the natural properties.
\end{remark} 
 
 \begin{remark}{\bf Spectral measure.}
 Recall that if $N$ is normal, then there exists (see e.g. Section 12 in [14]) a spectral measure $E$ from the $\sigma$-algebra of all Borel sets of $\sigma(N)$ into $B(H)$ such that if $\varphi$ is an essentially bounded Borel-measurable function on $\sigma(N)$ then 
 $$
\varphi(N)= \int_{\sigma(N)} \varphi \ dE.
$$
This could in an obvious way be used in defining $f_j(p(A))$, thus extending the  functional calculus even further.
\end{remark}


\subsection{Spectral mapping theorem for operators}

If $A\in \mathcal B(H)$ is such  that    $p(A)$ is similar to normal, then we have $\chi_A(f)= T\chi_V (f)T^{-1}$  and  therefore $\chi_A(f)$ and $\chi_V(f)$ have the same spectrum.  Therefore  we may as well assume that $A$ is polynomially normal.

 \begin{theorem}

Suppose $p$ has simple zeros and  $A\in \mathcal B(H)$ is such that $p(A)$ is normal.  Then for all $ [f] \in C_\Lambda(p(\sigma(A))) / \mathcal I_{\sigma(A)}$ we have
$$
\sigma(\chi_A (f)) = \sigma( [f]). 
$$
\end{theorem}

\begin{proof}
Recall that $\sigma([f])= \{\hat f(z) \ : \ z\in \sigma(A)\}.$  Consider first the inclusion
\begin{equation}\label{vaseninkl}
\hat f(z)  \in \sigma(\chi_A (f))  \  \text{ for all } z\in \sigma(A).
\end{equation}
where $f$  is of  the form 
\begin{equation}\label{erityismuoto}
f_j(w)= q_j(w, \overline w).
\end{equation}
 

 We take a $\lambda \in \sigma(A)$ and need to show that $\hat f(\lambda) \in \sigma(\chi_A (f))$.  The discussion splits into two as to whether
 \begin{equation}\label{aprox}
 \lambda \in \sigma_{ap}(A),
 \end{equation}
 or, if that is not the case, then necessarily,
 
 \begin{equation}\label{sittenpoint}
 \overline \lambda \in \sigma_p(A^*).
 \end{equation}
 Since $p(A)$ is normal we have in both cases $p(\lambda) \in \sigma_{ap}(p(A))$.
 
 Assuming (\ref{aprox}) there exists a sequence of unit vectors $x_n$ such that
 \begin{equation}\label{tahanviittaus}
 (A-\lambda)x_n \rightarrow 0
 \end{equation}
which by writing $p(A)-p(\lambda)=q(A,\lambda)(A-\lambda)$ implies immediately that
$$
(p(A)-p(\lambda))x_n \rightarrow 0.
$$
But then also 
$$
(p(A)-p(\lambda))^*x_n \rightarrow 0.
$$
In fact, if $N$ is normal and $Ny_n \rightarrow0$, then 
$$
(Ny_n, Ny_n) =  (N^*y_n, N^*y_n) \rightarrow 0.
$$
Denoting $p(A)=N$ and $p(\lambda)=\nu$ we have
\begin{align}
&\chi_A (f) -\hat f(\lambda)=\\
&\sum_{j=1}^d \delta_j(A) [ Q_j(N,N^*)-Q_j(\nu, \overline \nu)]
  +\\
  & \sum_{j=1}[\delta_j(A)-\delta_j(\lambda)]Q_j(\nu, \overline \nu).
 \end{align}
Operating with these at $x_n$ we have
$$
[Q_j(N,N^*)-Q_j(\nu,\overline \nu)]x_n \rightarrow 0
$$
since both $(N-\nu)x_n$ and $(N^*-\overline  \nu)x_n$ tend to 0. In fact,  there are polynomials $R, S$ of three variables such that we can write
$$
Q(N,N^*)-Q(\nu,\overline \nu)= [Q(N,N^*)-Q(\nu,N^*)]+[Q(\nu,N^*)-Q(\nu,\overline \nu)]
$$
$$
= R(N,\nu,N^*)(N-\nu) + S(\nu, N^*, \overline \nu)(N^*-\overline \nu)
$$

 Likewise, by (\ref{tahanviittaus}),
$
  [\delta_j(A)-\delta_j(\lambda)]Q_j(\nu, \overline \nu) x_n  \rightarrow 0,
$
 and so 
$
\hat f(\lambda) \in \sigma_{ap}(\chi_A (f)).
$

Next, assume that $\overline \lambda \in \sigma_p(A^*)$ and suppose $x$ is an eigenvector such that
 $$
A^*x = \overline\lambda x.
$$
Then clearly 
$$\{[\delta_j(A)-\delta_j(\lambda)]Q_j(\nu, \overline \nu)\}^* x=0.$$  However, we also have
$$
Q_j(N,N^*)^*x= \overline{Q_j(\nu, \overline\nu)}x
$$
since from $A^*x=\overline \lambda x$ we conclude $p(A)^*x = \overline{p(\lambda)}x$ and so $N=p(A)$ being normal this implies $Nx= \nu x$ as well.
Substituting these into $\chi_A (f)^*- \overline{\hat f(\lambda)}$ gives
$$
[\chi_A (f)^*- \overline{\hat f(\lambda)}]x=0.
$$
Hence $ \overline{\hat f(\lambda)} \in \sigma_p(\chi_A (f)^*)$ and so $\hat f(\lambda) \in \sigma(\chi_A (f))$.

We still need to show (\ref{vaseninkl}) when $\hat f$ is not of the form (\ref{erityismuoto}).  To that end assume that  $\hat f_n$ approximates $\hat f$ uniformly in $\sigma(A)$ where $\hat f_n$ is of the special form  (\ref{erityismuoto}).  

Take $\mu \in \hat f(\sigma(A))$ and we need to show that $\mu \in \sigma(\chi_A(f))$.  For some $\lambda\in \sigma(A)$ we thus have $\mu=\hat f(\lambda)$.
Let $\{\hat f_n\}$ be an approximative  sequence of the special form (\ref{erityismuoto})   such that in particular
$$
\sup_{z\in \sigma(A)} |\hat f(z)-\hat f_n(z)| \rightarrow 0
$$
and hence also
$$
\chi_A(f) = \lim_n \chi_A(f_n).
$$
Fix an arbitrary open set $V$ such that $\sigma(\chi_A(f)) \subset V$.  We show that $\mu \in V$ which completes the argument.  Fix an open set $U$ such that
$$
\sigma(\chi_A(f)) \subset U \subset {\rm cl}(U) \subset V.
$$
Since the spectrum is upper semicontinuous (e.g. Theorem 3.4.2 in [1]) there exists an $\varepsilon>0$ such that
$$
\sigma(B) \subset U   \  \text { whenever } \  \|\chi_A(f)-B\|<\varepsilon.
$$
Let $n_\varepsilon$ be such that $\|\chi_A(f)-\chi_A(f_n)\|< \varepsilon$  for all $n \ge n_\varepsilon.$  Then $\sigma(\chi_A(f_n)) \subset U$.  But for $\hat f_n$ we then have
$$
\hat f_n(\lambda) \in \sigma(\chi_A(f_n))\subset U.
$$
Finally,  from $\hat f_n(\lambda) \rightarrow \hat f(\lambda)$ we conclude that
$$
\mu= \hat f(\lambda) \in {\rm cl}(U) \subset V.
$$

\bigskip

Consider now the  other direction.  Here the conclusion follows easily  already from   Corollary \ref{tahanviitataan} with $K_0=\sigma(A)$. In fact, suppose $ \hat f(z)\not=0$ for $z\in \sigma(A).$    Then there exists $g \in C_\Lambda(\sigma(p(A)))$ such that 
$$
\hat f(z) \hat g(z)=1 \ \text{ for } z\in \sigma(A).
$$

By  Theorem \ref{peruslause} we then know that $[g]$ is the inverse of $[f]$ and since $\chi_A$ is a homomorphism from $C_\lambda(\sigma(p(A)))/\mathcal I_{\sigma(A)}$ to $\mathcal B(H)$, we have $$
\chi_A(f) \chi_A(g)=I
$$
and $0\notin \sigma(\chi_A(f))$.  
   Thus, if $\mu \in \sigma(\chi_A(f))$, then there must exist $\lambda \in \sigma(A)$ such that $\hat f(\lambda)-\mu =0$.  But this simply means that $\sigma(\chi_A(f)) \subset \hat f(\sigma(A))$.

\end{proof}

{ \bf References}
\smallskip

[1]  B.Aupetit, A Primer on Spectral Theory, Springer  1991

 [2] John B. Conway, A Course in Functional Analysis, Second Edition, Springer   (1990)
 
  [3] Eberhard Kaniuth,  A Course in Commutative Banach Algebras, Springer, 2009
  
  [4]  S. Foguel, Algebraic functions of normal operators, Israel J. Math., 6 (1968), 199-201
  
  [5]   Theodore W. Gamelin, Uniform Algebras,  Prentice-Hall, Inc,1969
  
   [6] F.R. Gantmacher, The Theory of Matrices, Volume one. AMS,  1959
   
   [7]  Frank Gilfeather, Operator valued roots of Abelian analytic functions,
 Pacific J. of  Mathematics,  Vol. 55, No. 1, 1974,  127- 148
 
  [8]  Fuad Kittaneh, On the structure of polynomially normal operators, Bull. Austr. Math. Soc. Vol. 30 (1984), 11-18
  
  [9] O. Nevanlinna, Multicentric Holomorphic Calculus, Computational Methods and Function Theory,  June 2012, Volume 12, Issue 1, pp 45-65
  
  [10] O.Nevanlinna,  Lemniscates and K-spectral sets, J. Funct. Anal. 262 (2012), 1728-1741
  
  [11]  Nicholas J. Higham,  Functions of matrices: theory and computation.  SIAM
2008

[12]  T. Ransford, Potentail Theory  in the Complex Plane, London Math. Soc. Student Texts {\bf 28}, Cambridge Univ. Press, 1995

[13]  Ch.E.Rickart, General Theory of Banach Algebras, D.Van Nostrand Company, inc. 1960

[14]   Walter Rudin, Functional Analysis, McGraw-Hill, 1973

\bigskip
{\bf Acknowledgement} 
\smallskip

Much of this work was written  during  2014 while the author was a Visiting Fellow at Clare Hall, University of Cambridge.  The visit was partially supported by The Osk. Huttunen Foundation.

\end{document}